\documentclass[11pt,reqno]{amsart}

\title[Instability for the rotation set of torus diffeomorphisms]{Instability for the rotation set of diffeomorphisms of the torus homotopic to the identity}
\author{Pierre-Antoine Guihéneuf}
\address{Universidade Federal Fluminense, Niterói}
\email{pierre-antoine.guiheneuf@math.u-psud.fr}

\usepackage[latin1]{inputenc}
\usepackage[english]{babel}
\usepackage[T1]{fontenc}
\usepackage{amsfonts}
\usepackage{amsmath}
\usepackage{amssymb}
\usepackage{stmaryrd}
\usepackage{amsthm}
\usepackage{enumerate}
\usepackage{pgf,tikz}
\usetikzlibrary{decorations.pathreplacing, shapes.multipart, arrows, matrix}
\usepackage[margin=0pt]{caption}
\usepackage[pdftex,colorlinks=true,linkcolor=blue,citecolor=blue,urlcolor=blue]{hyperref}

\newtheorem{lemme}{Lemma}
\newtheorem{theoreme}[lemme]{Theorem}

\theoremstyle{definition}

\theoremstyle{remark}
\newtheorem{rem}[lemme]{Remark}

\newcommand{\N}{\mathbf{N}}

\newcommand{\R}{\mathbf{R}}

\newcommand{\T}{\mathbf{T}}

\newcommand{\Q}{\mathbf{Q}}
\newcommand{\Z}{\mathbf{Z}}

\begin{document}

\maketitle

%

\begin{abstract}
The aim of this short note is to explain how the arguments of the ``closing lemma with time control'' of F. Abdenur and S. Crovisier \cite{MR2975581} can be used to answer Question 1 of the article ``Instability for the rotation set of homeomorphisms of the torus homotopic 
to the identity'' of S. Addas-Zanata \cite{MR2054045}.
\end{abstract}

\bigskip

In this short note, we explain how to get a $C^1$ version of a perturbation result of the rotation set of homeomorphisms of the torus homotopic to the identity, obtained by S. Addas-Zanata in \cite{MR2054045}: consider some diffeomorphism $f$ of the torus, isotopic to the identity, and suppose that some extreme point $(t,\omega)$ of the rotation set of $f$ has at least one irrational coordinate. Then there exists a perturbation $g$ of $f$, which is arbitrarily $C^1$-close to $f$, such that the rotation set of $g$ contains some vector that was not in the rotation set of $f$.

We will use the notations of \cite{MR2054045}. Let us recall the most useful ones: we will denote $\T^2 = \R^2/\Z^2$ the flat torus. The space $D^1(\T^2)$ will be the set of $C^1$-diffeomorphism of the torus $\T^2$ homotopic to the identity, endowed with the classical $C^1$ topology on compact spaces; $D^1(\R^2)$ will be the set of lifts to the plane of elements of $D^1(\T^2)$. Given $\tilde f\in D^1(\R^2)$, its \emph{rotation set} will be defied as 
\[\rho(\tilde f) = \bigcap_{i=1}^\infty \overline{\bigcup_{n\ge i}\Big\{ \frac{\tilde f^n(\tilde x)-\tilde x}{n}\mid \tilde x\in\R^2\Big\}}.\]
For $\tilde x\in\R^2$, we will denote
\[\rho(\tilde x,n,\tilde f) = \frac{\tilde f^n(\tilde x)-\tilde x}{n}\]
the rotation vector of the segment of orbit $\tilde x,\tilde f(\tilde x),\cdots,\tilde f^n(\tilde x)$, and when it is well defined (for example for a periodic point),
\[\rho(\tilde x,\tilde f) = \lim_{n\to +\infty} \rho(\tilde x,n,\tilde f).\]
We will also consider $\omega$ a volume or a symplectic form on $\T^2$, whose lift to $\R^2$ will also be denoted by $\omega$.

We will prove the following result.

\begin{theoreme}\label{Th1}
Let $\tilde f \in D^1(\R^2)$ be such that $\rho(\tilde f)$ has an extremal point $(t,\omega)\notin\Q^2$. Then there exists $\tilde g \in D^1(\R^2)$, arbitrarily $C^1$-close to $f$, such that $\rho(\tilde g)\cap \rho(\tilde f)^\complement \neq \emptyset$ (and in particular, $\rho(\tilde g) \neq \rho(\tilde f)$).

Moreover, if $\tilde f$ preserves $\omega$, then $\tilde g$ can be supposed to preserve it too. 
\end{theoreme}

We will prove this theorem by replacing the $C^0$ perturbation result of \cite{MR2054045} by a closing lemma in topology $C^1$, obtained by adapting the arguments of Theorem~6 of \cite{MR2975581}.

\begin{lemme}[Closing lemma with rotation control]\label{Lem1}
Let $\tilde f \in D^1(\T^2)$, $L : \R^2 \to \R$ a non-trivial affine form, and $\mathcal V$ a $C^1$-neighbourhood of $f$. Then, there exists $N\in N$ such that for every non-periodic point $x$ of $f$, there exists a neighbourhood $V$ of $x$ such that if $n\ge N$ and $y\in V$ are such that $f^n(y)\in V$ and $L\big(\rho(\tilde y,n,\tilde f)\big)>0$, then there exists $g\in\mathcal V$ such that $y$ is a periodic point\footnote{Note that in general, this period is different from $n$.} of $g$ satisfying $L\big(\rho(\tilde y,\tilde g)\big)>0$. Moreover, if $f$ preserves $\omega$, then $g$ can be supposed to preserve it too. 
\end{lemme}

The idea of the proof of this lemma is identical to that of Theorem 6 of \cite{MR2975581}, by replacing the dichotomy ``$\ell$ divides / does not divide the length of the orbit'' by the dichotomy ``$L\big(\rho(\tilde x,n,\tilde f)\big)>0$ / $L\big(\rho(\tilde x,n,\tilde f)\big)\le 0$''. More precisely, the proof of the connecting lemma of S. Hayashi \cite{MR1432037} builds a ``closable'' pseudo-orbit\footnote{A pseudo-orbit is called \emph{closable} if Pugh's algebraic lemma (Lemma 4 of \cite{MR2975581}, see also \cite{MR0226669}) can be applied simultaneously to every jump of the pseudo-orbit, to make it become a real orbit.} from a recurrent orbit of $f$, by making \emph{shortcuts} in this orbit; each time such a shortcut is performed there are two possibilities of creating a new pseudo-orbit (see Figure~\ref{Fig1}). If the initial orbit belongs to the set $\{L(\rho)>0\}$, then at least one of these two new pseudo-orbits also belongs to the set $\{L(\rho)>0\}$ (as the rotation vector of the initial orbit is a barycentre of the two new ones)\footnote{This corresponds to the initial argument of \cite{MR2975581}: ``If $\ell$ does not divide the length of the initial orbit, then it also does not divides the length of at least one of these two new pseudo-orbits''.}.

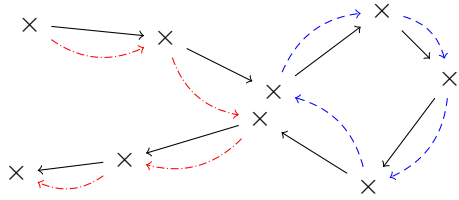
\begin{figure}
\begin{tikzpicture}[scale=1.8]	
\node (A) at (.3,0) {$\times$};
\node (B) at (1.3,-.1) {$\times$};
\node (C) at (2.1,-.5) {$\times$};
\node (D) at (2.9,.1) {$\times$};
\node (E) at (3.4,-.4) {$\times$};
\node (F) at (2.8,-1.2) {$\times$};
\node (G) at (2,-.7) {$\times$};
\node (H) at (1,-1) {$\times$};
\node (I) at (.2,-1.1) {$\times$};

\draw[->] (A) to (B);
\draw[->] (B) to (C);
\draw[->] (C) to (D);
\draw[->] (D) to (E);
\draw[->] (E) to (F);
\draw[->] (F) to (G);
\draw[->] (G) to (H);
\draw[->] (H) to (I);

\draw[->, red, densely dashdotted] (A) to[bend right] (B);
\draw[->, red, densely dashdotted] (B) to[bend right] (G);
\draw[->, red, densely dashdotted] (G) to[bend left] (H);
\draw[->, red, densely dashdotted] (H) to[bend left] (I);

\draw[->, blue, densely dashed] (C) to[bend left]  (D);
\draw[->, blue, densely dashed] (D) to[bend left] (E);
\draw[->, blue, densely dashed] (E) to[bend left] (F);
\draw[->, blue, densely dashed] (F) to[bend right] (C);

\end{tikzpicture}
\caption{If the rotation vector of the initial orbit (in black) is in $\{L>0\}$, then the rotation vector of one of the two pseudo orbits (in red and in blue) too.}\label{Fig1}
\end{figure}

\begin{proof}[Proof of Lemma~\ref{Lem1}]
Simply remark that Proposition~4 of \cite{MR2975581} still holds when condition
\begin{enumerate}
\item[3.] The length of the periodic pseudo-orbit $(y_1,\cdots,y_n = y_0)$ is not a multiple of $\ell$.
\end{enumerate}
is replaced by the condition
\begin{itemize}
\item[3.] The periodic pseudo-orbit\footnote{To be rigorous here, pseudo-orbits must be considered in the cover $\R^2$ and perturbations of diffeormorphisms performed in $\T^2$.} $(y_1,\cdots,y_n = y_0)$ satisfies $L\big(\frac{\tilde y_n - \tilde y_0}{n}\big)>0$.
\end{itemize}
The rest of the proof is identical to Section~3.3.1 of \cite{MR2975581}.
\end{proof}

We now explain how this connecting lemma with rotation control can be applied to adapt the proof of Theorem 1 of \cite{MR2054045} to the $C^1$ case. Let us quickly recall the main arguments of the proof in the $C^0$ case. As the rotation set is convex \cite{MR1053617}, there exists a supporting line of $\rho(\tilde f)$ at $(t,\omega)$, in other words an affine map $L : \R^2\to \R$ such that $L(t,\omega)=0$ and $L(v)\le 0$ for every $v\in\rho(\tilde f)$. Thus, if we build $g$ close to $f$ such that there exists $v\in\rho(\tilde g)$ satisfying $L(v)>0$, then we are done.

The ergodic theorem implies the existence of a point $x_0\in\T^2$ which is recurrent for $f$ and such that $\rho(\tilde x_0,\tilde f) = (t,\omega)$. At this point there are two possibilities. Either there exists $n$ arbitrarily large such that $f^n(x_0)$ is close to $x_0$ and $L(\rho(\tilde x_0, n, \tilde f)\big)>0$; in this case it suffices to apply a $C^0$ closing lemma to $x_0$ and $f^n(x_0)$ to get the theorem. Or for every $n$ large enough such that $f^n(x_0)$ is close to $x$, we have $L(\rho(\tilde x_0, n, \tilde f)\big)\le 0$. This case is a bit more complicated: we begin by proving that in this case, it is possible to suppose that $L(\rho(\tilde x_0, n, \tilde f)\big)< 0$ (Lemma~3 of \cite{MR2054045}). Let $n_0$ be such a number (large enough); a theorem of recurrence of G. Atkinson \cite{MR0419727} implies the existence of a time $n_1\gg n_0$ such that $L(n_1\rho(\tilde x_0, n_1, \tilde f)\big)$ is arbitrarily close to 0. A calculation shows that in this case, $L\big(\rho(\tilde f^{n_0}(\tilde x_0), n_1-n_0, \tilde f)\big)>0$: the rotation vector of the segment of orbit between $\tilde f^{n_0}(\tilde x_0)$ and $\tilde f^{n_1}(\tilde x_0)$ belongs to $\{L>0\}$. It then suffices to apply the $C^0$ closing lemma to $\tilde f^{n_0}(\tilde x_0)$ and $\tilde f^{n_1}(\tilde x_0)$.

\begin{proof}[Proof of Theorem~\ref{Th1}]
Let $\tilde f \in D^1(\R^2)$ be such that $\rho(\tilde f)$ has an extremal point $(t,\omega)\notin\Q^2$, and $\mathcal V$ a $C^1$-neighbourhood of $f$. We fix once for all a lift $\tilde f$ of $f$, and choose $L : \R^2 \to \R$ an affine form such that $L(t,\omega)=0$ and $L(v)\le 0$ for every $v\in\rho(\tilde f)$. Let $x_0\in\T^2$ be a recurrent point of $f$ such that $\rho(\tilde x_0,\tilde f) = (t,\omega)$. Lemma~\ref{Lem1} gives us a number $N\in \N$ and a neighbourhood $V$ of $x_0$. The proof of Theorem~1 of \cite{MR2054045} summarized in the previous discussion gives us a point $y=f^{n_0}(x_0)$ (with $n_0$ possibly equal to 0) and a time $n_1\ge N$ such that $\tilde f^{n_1}(\tilde y)\in V$ and $L\big(\rho(\tilde y,n_1,\tilde f)\big)>0$. Applying Lemma~\ref{Lem1}, we get $g\in\mathcal V$ such that $y$ is a periodic point of $g$ satisfying $L\big(\rho(\tilde y,\tilde g)\big)>0$. This proves the theorem.

Moreover, if $f$ preserves $\omega$, then $g$ can be supposed to preserve it too.
\end{proof}

\begin{rem}
Theorem~1 of \cite{MR2054045} is also true in the $C^0$ measure-preserving case. To see it, it suffices to replace the $C^0$ closing lemma by the measure-preserving one (see for example Lemma 13 of \cite{Oxto-meas}).
\end{rem}

\bibliographystyle{amsalpha}
\bibliography{../../Biblio}

\providecommand{\MR}[1]{}
\providecommand{\bysame}{\leavevmode\hbox to3em{\hrulefill}\thinspace}
\providecommand{\MR}{\relax\ifhmode\unskip\space\fi MR }
\providecommand{\MRhref}[2]{%
  \href{http://www.ams.org/mathscinet-getitem?mr=#1}{#2}
}
\providecommand{\href}[2]{#2}
\begin{thebibliography}{Hay97}

\bibitem[AC12]{MR2975581}
Flavio Abdenur and Sylvain Crovisier, \emph{Transitivity and topological mixing
  for {$C^1$} diffeomorphisms}, Essays in mathematics and its applications,
  Springer, Heidelberg, 2012, pp.~1--16. \MR{2975581}

\bibitem[Atk76]{MR0419727}
Giles Atkinson, \emph{Recurrence of co-cycles and random walks}, J. London
  Math. Soc. (2) \textbf{13} (1976), no.~3, 486--488. \MR{0419727}

\bibitem[AZ04]{MR2054045}
Salvador Addas-Zanata, \emph{Instability for the rotation set of homeomorphisms
  of the torus homotopic to the identity}, Ergodic Theory Dynam. Systems
  \textbf{24} (2004), no.~2, 319--328. \MR{2054045 (2005c:37074)}

\bibitem[Hay97]{MR1432037}
Shuhei Hayashi, \emph{Connecting invariant manifolds and the solution of the
  {$C^1$} stability and {$\Omega$}-stability conjectures for flows}, Ann. of
  Math. (2) \textbf{145} (1997), no.~1, 81--137. \MR{1432037}

\bibitem[MZ89]{MR1053617}
Micha{\l} Misiurewicz and Krystyna Ziemian, \emph{Rotation sets for maps of
  tori}, J. London Math. Soc. (2) \textbf{40} (1989), no.~3, 490--506.
  \MR{1053617 (91f:58052)}

\bibitem[OU41]{Oxto-meas}
John Oxtoby and Stanislaw Ulam, \emph{Measure-preserving homeomorphisms and
  metrical transitivity}, Ann. of Math. \textbf{42} (1941), no.~2, 874--920.

\bibitem[Pug67]{MR0226669}
Charles Pugh, \emph{The closing lemma}, Amer. J. Math. \textbf{89} (1967),
  956--1009. \MR{0226669 (37 \#2256)}

\end{thebibliography}

\end{document}